\documentclass[12pt]{amsart}

\usepackage{fullpage}
\usepackage{mathrsfs, amssymb,amsthm, amsfonts, amsmath}
\usepackage[all]{xy}
\usepackage{upgreek}
\usepackage{amsmath}

\newtheorem{theorem}[equation]{Theorem}
\newtheorem*{theorem*}{Theorem}

\newtheorem{question}[equation]{Question}

\theoremstyle{definition}

\newtheorem*{definition*}{Definition}

\theoremstyle{remark}

\newcommand{\QQ}{\mathbb{Q}}

\newcommand{\PP}{\mathbb{P}}
\newcommand{\KK}{\mathbb{K}}
\newcommand{\LL}{\mathbb{L}}

\newcommand{\mumu}{{\boldsymbol{\mu}}}

\newcommand{\Aut}{\operatorname{Aut}}

\newcommand{\PGL}{\operatorname{PGL}}

\newcommand{\Gal}{\operatorname{Gal}}

\date{}

\title{Non-abelian groups acting on Severi--Brauer surfaces}

\author{Constantin Shramov}

\address{Steklov Mathematical Institute of Russian Academy of Sciences, 8 Gubkina st.,
Moscow, 119991, Russia}

\email{costya.shramov@gmail.com}

\begin{document}

\begin{abstract}
We provide examples of finite non-abelian groups acting on non-trivial Severi--Brauer surfaces.
\end{abstract}

\maketitle

A \emph{Severi--Brauer surface} over a field $\KK$ is a surface that becomes isomorphic to $\PP^2$ over the algebraic closure of~$\KK$.
Concerning birational automorphisms of non-trivial Severi--Brauer surfaces (i.e. those that are not isomorphic to~$\PP^2$)
the following is known.

\begin{theorem}[{\cite[Theorem~1.2(ii)]{Shramov-SB}}]
\label{theorem:J-3}
Let $S$ be a non-trivial Severi--Brauer surface over a field $\KK$ of characteristic zero.
Then every finite group acting by birational automorphisms of $S$ is either abelian, or contains a normal abelian subgroup
of index~$3$.
\end{theorem}

The purpose of this note is to prove the following. We denote by $\mumu_n$ the cyclic group of order $n$.

\begin{theorem}\label{theorem:attained}
Let $p=3k+1$ be a prime number. Consider the non-trivial
homomorphism~\mbox{$\mumu_3\to\mumu_{3k}\cong\Aut(\mumu_p)$}, and let $G_p=\mumu_p\rtimes\mumu_3$ be the corresponding semi-direct product. Then
there exists a number field $\KK$ and a non-trivial Severi--Brauer surface $S$ over~$\KK$ such that
the group $\Aut(S)$ contains the group~$G_p$.
\end{theorem}

Using the notion of the \emph{Jordan constant} 
(see \cite[Definition~1]{Popov}), 
one can reformulate Theorem~\ref{theorem:J-3} by saying that the Jordan constant of the birational
automorphism group (and thus also of the usual automorphism group) of a non-trivial Severi--Brauer surface over a field of characteristic zero is at most~$3$.
In these terms, Theorem~\ref{theorem:attained} states that for a suitable number field $\KK$ and a suitable
non-trivial Severi--Brauer surface $S$ over~$\KK$ the Jordan constant of the automorphism group of $S$
(and thus also of its birational automorphism group) attains this bound.

\begin{proof}[Proof of Theorem~\ref{theorem:attained}]
Let $\xi$ denote a non-trivial root of unity of degree~$p$, 
and let~\mbox{$\LL=\QQ(\xi)$}.
Then $\LL/\QQ$ is a Galois extension with
the Galois group isomorphic to $\mumu_{p-1}$, see for instance  \cite[Lemma~III.1.1]{CasselsFrolich}.
Let $\KK\subset\LL$ be the field of invariants of the (unique) subgroup~\mbox{$\mumu_{3}\subset\Gal(\LL/\QQ)$}.
Then $\LL/\KK$ is a Galois extension with
the Galois group $\Gal(\LL/\KK)$ isomorphic to~$\mumu_3$.
A generator $\sigma$ of $\Gal(\LL/\KK)$ sends $\xi\in\LL$ to $\xi^d$, where $d\neq 1$ is an integer such that $d^3\equiv 1\pmod p$.

There exists an element $a\in\KK$ such that $a$
is not contained in the image of the Galois norm of the field extension $\LL/\KK$, see \cite[Theorem~1(b)]{Stern}.
We consider the cyclic algebra~$A$ over $\KK$ associated with $\sigma$ and $a$,
see \cite[\S2.5]{GilleSzamuely} or \cite[Exercise~3.1.6]{GorchinskyShramov}.
In other words, $A$ is generated over $\KK$ by $\LL$ and an element $\alpha$ subject to relations
$\alpha^3=a$ and
\begin{equation*}
\lambda \alpha=\alpha\sigma(\lambda), \quad \lambda\in\LL,
\end{equation*}
so that in particular we have
\begin{equation*}
\xi \alpha=\alpha\xi^d.
\end{equation*}
Then $A$ is a central simple algebra. However, $A$ is not a matrix algebra, see for instance~\mbox{\cite[Exercise~3.1.6(i)]{GorchinskyShramov}}.
Since the dimension of $A$ over $\KK$ is $9$, we conclude from the theorem of Wedderburn (see \cite[Theorem~2.1.3]{GilleSzamuely}) that
$A$ is a division algebra.

The order of $\alpha$ in the multiplicative group $A^*$ equals $3$; since $\alpha\not\in \KK$, the order
of its image $\hat{\alpha}$ in the quotient group $A^*/\KK^*$ equals $3$ as well.
Similarly, one has $\xi^p\in\KK$, and since $\xi\not\in\KK$, we conclude that for any $r<p$
one has $\xi^r\not\in\KK$. Thus the order of the image~$\hat{\xi}$ of $\xi$ in $A^*/\KK^*$  equals $p$.
Furthermore, we have
$$
\hat{\xi}\hat{\alpha}=\hat{\alpha}\hat{\xi}^d.
$$
These are exactly the defining relations of the group $G_p$. This means that the group $\hat{G}_p$
generated by $\hat{\xi}$ and $\hat{\alpha}$ in $A^*/\KK^*$ is a quotient of $G_p$ such that $\hat{G}_p$ contains an element
of order $p$ and an element of order~$3$. Hence $\hat{G}_p\cong G_p$.

We see that the elements $\hat{\xi}$ and $\hat{\alpha}$  generate a subgroup isomorphic to $G_p$ in $A^*/\KK^*$.
It remains to recall that for a Severi--Brauer surface $S$ corresponding to
$A$ one has~\mbox{$\Aut(S)\cong A^*/\KK^*$}, see for instance
\cite[Lemma~4.1]{ShramovVologodsky}.
\end{proof}

It would be interesting to obtain a complete classification 
of finite groups acting on non-trivial Severi--Brauer surfaces, 
similarly to what was done for conics in~\cite{Garcia-Armas}.
Also, we point out that for a given number field~$\KK$ there is only a finite number of finite groups 
$G$ such that there exists a Severi--Brauer surface over $\KK$ with an action of~$G$, cf.~\mbox{\cite[Theorem~1.4]{Prokhorov-Shramov-J}}. This follows from the observation 
that any group $G$ like this acts faithfully on the anticanonical linear system of the Severi--Brauer surface, and thus is embedded into~\mbox{$\PGL_{10}(\KK)$}.
On the other hand, 
the latter group contains only a finite number of finite subgroups up to 
isomorphism by the (generalized) theorem of Minkowski, see for instance~\cite[Theorem~5]{Serre-07}.  
I do not know the answer to the following question.

\begin{question}
Does there exist an (infinite) extension $\KK'$ of $\QQ$ and a non-trivial Severi--Brauer surface $S'$
over $\KK'$ such that the group $\Aut(S)$ contains all the groups~$G_p$?
\end{question}

\textbf{Acknowledgements.}
I am grateful to S.\,Gorchinskiy and D.\,Osipov for useful discussions.


\end{document}